\documentclass[12pt]{article}
\usepackage{amsfonts}
\usepackage{amsmath, amsthm, upref}
\usepackage{enumerate}
\usepackage[parfill]{parskip}
\usepackage{color}
\usepackage{hyperref}
\usepackage{multimedia}
\usepackage{graphics}
\usepackage{epsfig}
\usepackage{graphicx}
\usepackage[greek,english]{babel}
\usepackage{textcomp}
\usepackage{eurosym}

\DeclareMathOperator{\e}{e}

\newcommand{\gl}{\lambda}

\newcommand{\mbf}{\mathbb{F}}
\newcommand{\mbg}{\mathbb{G}}
\newcommand{\comment}[1]{}

\newcommand{\mbr}{\mathbb{R}}

\newcommand{\mcf}{\mathcal{F}}

\newcommand{\mcg}{\mathcal{G}}

\newcommand{\ga}{\alpha}
\newcommand{\gs}{\sigma}

\newtheorem{theorem}{Theorem}

\newtheorem{assumption}[theorem]{Assumption}

\newtheorem{lemma}{Lemma}

\newtheorem{proposition}{Proposition}
\newtheorem{remark}{Remark}

\newcommand{\bee}{\begin{equation}}
\newcommand{\eee}{\end{equation}}
\newcommand{\bea}{\begin{eqnarray}}
\newcommand{\eea}{\end{eqnarray}}
\newcommand{\bean}{\begin{eqnarray}}
\newcommand{\eean}{\end{eqnarray}}

\newcommand{\raw}{\rightarrow}

\newcommand{\bgt}{\bar{\tau}}
\newcommand{\tY}{\tilde{Y}}
\newcommand{\bga}{\bar{\ga}}
\newcommand{\sign}{\text{\emph{sign}}}
\newcommand{\oneut}{\frac{1}{U_t}}

\begin{document}

\title{Credit Valuation Adjustment in Credit Risk with Simultaneous Defaults Possibility}
\author{Aditi Dandapani
\thanks{Institut f\"ur Mathematik, Universit\"at Z\"urich, Winterthurerstrasse 190, CH-8057 Z\"urich; email: ad2259@caa.columbia.edu
}
 ~and Philip Protter
\thanks{Statistics Department, Columbia University, New York, NY 10027; email: pep2117@columbia.edu.
}
\thanks{Supported in part by NSF grant DMS-1612758}
}
\date{\today }
\maketitle

\begin{abstract}

In a series of recent papers, Damiano Brigo, Andrea Pallavicini, and co-authors have shown that the value of a contract in a Credit Valuation Adjustment (CVA) setting, being the sum of the cash flows, can be represented as a solution of a decoupled forward-backward stochastic differential equation (FBSDE). CVA is the difference between the risk-free portfolio value and the true portfolio value that takes into account the possibility of a counter party's default. In other words, CVA is the market value of counter party credit risk. This has achieved noteworthy importance after the 2008 financial debacle, where counter party risk played an under-modeled but huge risk. In their analysis, Brigo et al make the classical assumption of conditional independence of the default times, given the risk free market filtration. This does not allow for the increasingly likely case of simultaneous defaults. We weaken their assumption, replacing it with a martingale orthogonality condition. This in turn changes the form of the BSDE that arises from the model. 
 
\end{abstract}

\section{Introduction}\label{intro}
In the traditional theory of financial derivatives one does not take into account the default risk of the contracting parties. Such risks have not been ignored in the literature, to be sure, see for example Bielecki and Rutkowski (2002) or Brigo and Masetti (2005), or more recently Bichuch, Capponi, and Sturm (2016), but they had not been much emphasized in the literature. All that changed after the financial debacle and collapse of 2008, which is usually traced to the bankruptcy and collapse of Bear Stearns  and -- especially -- of Lehmann Brothers, both investment houses previously considered to be `too big to fail." In the ashes of the crisis rose the theory of Credit Valuation Adjustment, known by its acronym CVA. It incorporates the adjustment to the pricing of a financial product that is due to the default risk of the counterparty. 

An example is the purchase of a financial product such as a call option from a counterparty we will call ``CP." If CP defaults, then one might suffer losses due to the inability of the CP to deliver the option's payoff. These types of losses were typically not considered in the general theory before the 2008 debacle. Often CDO's and the like were protected via insurance from monoline insurance companies, but these insurance companies typically were undercapitalized and failed along with the counter parties to the CDO's. It became clear that these risks (of counter party failure) were serious risks, and needed to be taken into account.

These issues have been treated in a series of papers by D. Brigo and A. Pallavicini and co-authors, culminating in the paper Brigo, Francischello, and Pallavicini (2019). They give a natural formulation of a contract's value, expressed as the (unique) solution of a forward-backward stochastic differential equation (FBSDE). An hypothesis common to virtually all of the previous work (with the notable exception of the work of Cr\'epey, Jeanblanc, and Zagari (2009) is to have default stopping times compatible with Cox constructions, with the exponentially distributed random variables in the Cox constructions  being independent of an underlying common filtration $\mbf=(\mcf_t)_{t\geq 0}$. This creates a relatively standard hypothesis of conditional independence given $\mbf$ of the two default times in question in the standard model. 

The goal of this paper is to weaken the hypothesis of conditional independence. We feel this is an advance, because it better allows for credit contagion. We will elaborate on this point in Section~\ref{xx}. 

\section{The Basic Framework}

We begin with a default-free filtration, $\mathbb{F}$, which is generated by the underlying asset, $S=(S_t)_{t\geq 0},$ which in turn satisfies, under the objective measure $P$:

\begin{equation}\label{e0} 
dS_t=r_tS_tdt+\sigma(S_t,t)dW_t. 
\end{equation} 
In the above, the $\mathbb{F}$ adapted process $r_t$ is the risk free interest rate.  We are also given a smooth function $\Phi$ and we will work with a claim of the form $\Phi(S_T).$

We will recall the setting in the work of Brigo, Francichello, and Pallavicini (2018) that we use as our starting point. We will refer to this paper as [BFP] hereafter. The goal is to price a generic derivative contract, with maturity $T$, between two financial entities, an investor $I$ and a counterparty $C$. It is often convenient to think of $I$ as representing the interests of a bank. The authors of [BFP] then use the standard reduced form framework, constructing two possible default times, $\tau_I$ and $\tau_C$, via Cox constructions. The arrival intensities of $\tau_I$ and $\tau_C$ are assumed to be adapted to an underlying filtration $\mbf=(\mcf_t)_{t\geq 0}$. The larger filtration $\mbg$ includes the independent exponential random variables $Z$ that are used in the Cox construction. 

A basic assumption, and a classical one, in the article [BFP] is that \emph{the two stopping times $\tau_I$ and $\tau_C$ are conditionally independent given the filtration $\mbf$}. 

\bigskip

In this paper, we wish to weaken this assumption. 

\bigskip

Using the conditional independence assumption, the authors of [BFP] derive a Forward-Backwards Stochastic Differential Equation. To wit, they define the Value process $V$ by 
\bee\label{eq1}
V_t=F_t+C_t+H_t+C^H_t.
\eee
which is the sun of the funding, collateral, and hedging accounts. The [BFP] authors next assume that the hedge is perfectly collateralized, and hence the amount the bank needs to finance at a given time $t$ is given by:
$$
F_t=V_t-C_t.
$$
To find the value $V_t$, [BFP] compute the conditional expectation under a risk neutral measure $Q$ of the sum of all future cash flows, discounted at the risk-free rate. Since we work on a finite horizon set-up, we let $\bgt=\tau\wedge T$. In this setting, we can write the equation for the value process $V$ as follows:
\begin{eqnarray}\label{eq2}
V _t&=&E\{1_{\{\tau>T\}}D(t,T)\Phi(S_T)+\int_t^TD(t,u)\pi_udu\vert\mcg_t\}\bigskip\bigskip\bigskip\text{ (contractual cash flows)}\notag\\
&-&E\{\int_t^{\bgt} D(t,u)(c_u-r_u)C_udu\vert\mcg_t\}\bigskip\bigskip\text{ (carrying cost of collateral account)}\notag\\
&-&E\{\int_t^{\bgt} D(t,u)(f_u-r_u)(V_u-C_u)du\vert\mcg_t\}\bigskip \bigskip \text{ (cost of funding the account)}\notag\\
&-&E\{\int_t^{\bgt} D(t,u)(r_u-h_u)H_udu\vert\mcg_t\}\bigskip\bigskip\text{ (Hedging costs)}\notag\\
&+&E\{D(t,\tau)1)_{\{t<\tau<T\}}(\theta_\tau + 1_{\{b\}}1_{\{\tau=\tau_j\}}LGD_t(V_{\tau_j}-C_{\tau_j})^+)\}
\end{eqnarray}
where we need some more explanations of the notation used in equation~\eqref{eq2} above. We have the following:
\begin{enumerate}
\item $c$ is $\mbf$ predictable and represents the collateral remuneration rate
\item $r$ is the spot interest rate and is $\mbf$ predictable
\item $f$ is the funding rate, and it is $\mbg$ adapted
\item $h$ is the hedging rate, and it is $\mbg$ adapted
\item $D(s,t)=D(s,t,r)=e^{-\int_s^tr_udu}$, the discount factor associated to the rate $r$
\item $(LGD_t)_{t\geq 0}$ denotes the Loss Given Default process as it evolves in time
\item Capital letters represent the antiderivatives, or cumulative quantities. E.g., $H_t=\int_0^t h_udu$
\end{enumerate}

Recall the elementary but fundamental version of a Bayes' result:

\begin{lemma}\label{ell1}
For a Borel measurable random variable $X$ and for any $t\in\mbr_+$, we have
\bee\label{eq3}
E\{1_{\{t<\tau\leq u\}}X\vert\mcg_t\}=1_{\{\tau>t\}}\frac{E\{1_{\{t<\tau\leq u\}}X\vert\mcf_t\}}{E\{1_{\{\tau>t\}}\vert\mcf_t\}}
\eee
In particular, for any $\mbg$ random variable $Y$, there exists an $\mbf$ random variable $Z$ such that 
\bee\label{eq4}
1_{\{\tau>t\}}Y=1_{\{\tau>t\}}Z.
\eee
\end{lemma}

\section{The New Framework}\label{xx}

We discard the assumption that the two stopping times are conditionally independent given the filtration $\mathbb{F}.$ This can be interpreted as taking into account the dependency of  the default of one corporation on the other's default, and not just on economy-wide (or sector) market forces. 

We will use the following constructions of the two stopping times: Define the filtration $\mathbb{G}$ to be $\mathbb{G}_t=\mathbb{F}_t \vee \mathbb{H}^{1}_t \vee \mathbb{H}^{2}_t,$ where $\mathbb{H}^{1}_t=\sigma(\tau^{1} \wedge t)$ and $\mathbb{H}^{2}_t=\sigma(\tau^{2} \wedge t).$

 Let $\lambda^1$ and  $\lambda^2$ be $\mathbb{F}$ adapted processes.

Then, we define the following stopping times: \begin{equation*} \tau^{1}=\inf\{t: \int_{0}^{t}\lambda^{1}_sds \geq Z^{1}\} \end{equation*} 
 \begin{equation*} \tau^{2}=\inf\{t: \int_{0}^{t}\lambda^{2}_sds \geq Z^{2}\} \end{equation*}

Let $\tau$ be the minimum of the two stopping times. 

In the above, $(Z^{1},Z^{2})$ is a bivariate exponential variable, and is $\mathcal{G}$-measurable.  Instead of $Z^1$ and $Z^2$ being assumed to be independent both of $\mbf$ and of each other, we retain the independence from $\mbf$ but assume only  that $(Z^1,Z^2)$ is a bivariate exponential random vector. To wit we assume that it is has the following survival function:
\bee\label{e3}
P(Z^1>s,Z^2>t)=\exp\{-\ga^1s-\ga^2t-\ga^3\max(s,t)\}
\eee
and in keeping with the Cox construction paradigm we assume $\ga^1=\ga^2=1$, so that~\eqref{e3} becomes:
\bee\label{3bis}
P(Z^1>s,Z^2>t)=\exp\{-s-t-\bga\max(s,t)\}
\eee
where we have used $\bga$ instead of $\ga^3$ for notational clarity. 
We call such a bivariate exponential distribution a BVE.

\bigskip

\noindent There are some key properties of the BVE. First of all the marginals, ie the distributions of $Z^1$ and $Z^2$, are both exponentials of parameter 1. Second, and significant, is that the BVE does not have a density. Let $\mu=\ga^1+\ga^2+\ga^3$ which in our case becomes $\mu=2+\bga$. We then have
\bee
P(Z^1>x,Z^2 >y)=\frac{\ga^1+\ga^2}{\mu}F_a(x,y)+\frac{\ga^3}{\mu}F_s(x,y),
\eee
where $F_a$ denotes the absolutely continuous component of the tail cdf, and $F_s$ denotes the singular component. We have
\bee
F_s(x,y)=\exp\{-\mu\max(x,y)\}
\eee
which is singular with respect to Lebesgue measure on $\mbr^2$. We also have
\bee
F_a(x,y)=\frac{\mu}{\ga^1+\ga^2}\exp\{-\ga^1x-\ga^2y-\ga^3\max(x,y)\}-\frac{\ga^3}{\ga^1+\ga^2}\exp\{-\mu\max(x,y)\}
\eee

\noindent One key property of the BVE is that one can have $P(Z^1=Z^2)>0$. This models the situation where a major event could affect two separate companies simultaneously, as previously discussed. For more on BVE's one can consult, for example, the classic article of Marshall and Olkin (1967).

We now state and prove three lemmas that will prove integral to our ability to express the value of the contract in $V$ in equation~\eqref{eq2} in terms of the ``default-free" filtration $\mbf:$

\bigskip

\begin{lemma}\label{t4}
Let $\tau$ be a $\mbg$ stopping time, and let $\lambda_t$ be such that \begin{equation}\label{mart}M_t=1_{\{t\geq\tau\}}-\int_0^{t\wedge\tau}\gl_sds \end{equation} is a $\mbg$ martingale. Assume that $X$ is such that the martingale $Z_t=E(Xe^{-\int_ 0^T\gl_sds}\vert\mcg_t)$ is strongly orthogonal to $M$, in the $L^2$ sense (cf for example~\cite{PP} for the theory of strongly orthogonal martingales). Then
\bee\label{2e5}
E\{X1_{\{\tau>T\}}\vert\mcg_t\}=1_{\{\tau>t\}}E\{Xe^{-\int_t^T\gl_sds}\vert\mcg_t\}
\eee
\end{lemma}

\begin{proof}
Let $N_t=1_{\{t\geq\tau\}}$, and then 
$$
1-N_t=1_{\{t<\tau\}}\text{ and }1-N_{t-}=1_{\{t\leq\tau\}}.
$$
We define $L_t=(1-N_t)e^{\int_0^t\gl_sds}$, and we claim $L$ is a local martingale. Since $L$ is a semimartingale, we can use integration by parts to get:
\begin{eqnarray*}
dL_t&=&(1-N_{t-})d\left(e^{\int_0^t\gl_sds}\right)-e^{\int_0^t\gl_sds}dN_t\\
&=&(1-N_{t-})\gl_te^{\int_0^t\gl_sds}dt-e^{\int_0^t\gl_sds}dN_t\\
&=&e^{\int_0^t\gl_sds}\left((1-N_{t-})\gl_tdt-dN_t\right)\\
&=&e^{\int_0^t\gl_sds}(-dM_t)=-e^{\int_0^t\gl_sds}dM_t
\end{eqnarray*}
Since $M$ is a martingale (and even an $L^2$ martingale), we conclude first that $L$ is a local martingale; but since the integrand is bounded, we know that $L$ is an actual martingale, and indeed itself an $L^2$ martingale. We let
\bee\label{2e6}
Y_t=E(Xe^{-\int_t^T\gl_sds}\vert\mcg_t)
\eee
and we define $U$ by
$$
U_t=(1-N_t)Y_t
$$
By the assumption of the strong orthogonality of $M$ and $Z$, we have that $U$, where $U=LZ$, is a martingale. This gives:
\begin{eqnarray*}
E(U_T\vert\mcg_t)&=&U_t\text{ or equivalently}\\
E(X1_{\{\tau>T\}}\vert\mcg_t)&=&1_{\{\tau>t\}}E(Xe^{-\int_t^T\gl_sds}\vert\mcg_t)
\end{eqnarray*}
which was what to be shown.
\end{proof}

Within credit risk theory,  we assume a spot interest rate process $(r_s)_{0\leq s\leq T}$  and a cumulative interest rate is then $e^{-\int_t^Tr_udu}$, which is what a dollar at time 0 is worth at time $T$. \emph{Under the Cox process hypothesis} David Lando showed in 1998 that Lemma~\ref{t4} extends to the following:
\begin{equation}
E\left(e^{-\int_t^Tr_udu}X1_{\{\tau>T\}}\vert\mcg_t\right)=1_{\{\tau>t\}}E(Xe^{-\int_t^T(r_u+\gl_u)du}\vert\mcg_t) \end{equation} 
Lando's theorem extends to the more general case when we assume a Cox process construction, as the next result shows. In this article we prefer to assume an orthogonality condition, which is similar to independence, but  weaker as an assumption. This allows us to avoid the  complete markets trap deriving from results of El Karoui (1999)\footnote{ See Blanchet-Scalliet  and Jeanblanc (2004) for a detailed proof of El Karoui's result} and Kusuoka (1999): To wit, El Karoui has shown that Cox constructions arise naturally within a complete markets framework where Hypothesis H applies. Kusuoka, however, has shown that this framework can fall apart with a change to an equivalent risk neutral measure, thus \emph{inter alia} showing that one needs a complete market framework and Hypothesis H simultaneously. This makes the more general approach of orthogonal martingales more appealing. 

\bigskip

 \begin{lemma}\label{t5}
 Assume that $M$ given in~\eqref{mart} is strongly orthogonal to $Z$, where $Z_t=E(Xe^{-\int_ 0^T(r_s+\gl_s)ds}\vert\mcg_t)$.
Then we have:
 \bee\label{2e8}
E\{e^{-\int_t^Tr_udu}X1_{\{\tau>T\}}\vert\mcg_t\}=1_{\{\tau>t\}}E\{Xe^{-\int_t^T(r_u+\gl_u)du}\vert\mcg_t\}
\eee
\end{lemma}
 
 \begin{proof}
 The proof is identical to the proof of Theorem~\ref{t4} until we get to equation~\eqref{2e6}. There we change the definition of $Y$ (we call it now $\tY$ to distinguish it from the $Y$ in the proof of Lemma~\ref{t4}) to:
\bee\label{2e8}
\tY=E(Xe^{-\int_0^tr_sds}e^{-\int_t^T(r_s+\gl_s)ds}\vert\mcg_t)
\eee
Let $U_t=1_{\{\tau>t\}}\tY_t$. From the strong orthogonality of $M$ and $Z$, we have that $U=LZ$ is a martingale, where here $L_t=1_{\{\tau>t\}}e^{\int_0^t\gl_sds}$.
By the assumption of the strong orthogonality of both $M$ and $V$ with $Z$, we have that $L$ and $\tY$ are also strongly orthogonal. The result now follows, as in the previous proof.
 \end{proof}
 
 Lemma~\ref{t5} is useful for applications to credit risk. As an example, we consider a bond with continuous payout until default, or maturity of the bond, whichever comes first. Mathematically we can write the continuous payout in the form
$$
 g(X_s)1_{\{\tau>s\}}\text{ with cumulative payout from }t\text{ to }T\text{ as }\int_t^Tg(X_s)1_{\{\tau>s\}}ds.
$$
 With discounting, the cumulative payout is
 \bee\label{2e9}
 \int_t^Tg(X_s)1_{\{\tau>s\}}e^{-\int_t^sr_udu}ds
 \eee
 
 \begin{proposition}\label{t6}
 Under the continuous payout framework, and with the hypotheses of Lemma~\ref{t5} in force, we have
 \bee\label{2e10}
 E(\int_t^Tg(X_s)1_{\{\tau>s\}}e^{-\int_t^sr_udu}ds\vert\mcg_t)=1_{\{\tau>t\}}E(\int_t^Tg(X_s)e^{-\int_t^s(r_u+\gl_u)du}ds\vert\mcg_t)
 \eee
 \end{proposition}
 
 \begin{proof}
 In Lemma~\ref{t5} for fixed $s$ we let $g(X_s)$ play the role of the random variable $X$. Using Fubini's theorem for conditional expectations (see for example~\cite{PP}) and Lemma~\ref{t5} we have
 \begin{eqnarray*}
E(\int_t^Tg(X_s)1_{\{\tau>s\}}e^{-\int_t^sr_udu}ds\vert\mcg_t)&=&\int_t^TE(g(X_s)1_{\{\tau>s\}}e^{-\int_t^sr_udu}\vert\mcg_t)ds\\
&=&\int_t^T1_{\{\tau>t\}}E(g(X_s)e^{-\int_t^s(r_u+\gl_u)du}\vert\mcg_t)ds\\
&=&1_{\{\tau>t\}}E(\int_t^Tg(X_s)e^{-\int_t^s(r_u+\gl_u)du}\vert\mcg_t)ds
\end{eqnarray*} 
 \end{proof}

\section{Valuation of the contract $V$}\label{s4}

Let $V_t$ be the value of the contract (this is the sum of all discounted cash flows, which are defined in equation $1$ of Brigo at al (2019)). Then, we have:

 \begin{align}\label{contract} V_t&=\mathbb{E}[1_{\{\tau \textgreater T\}}\e^{-\int_{t}^{T}r_udu}\Phi(S_T)|\mathcal{G}_t]+ \mathbb{E}[\int_{t}^{T \wedge \tau}\e^{-\int_{t}^{u}r_sds}A_udu|\mathcal{G}_t]+ \\& \nonumber \mathbb{E}[\e^{-\int_{t}^{\tau}r_sds}1_{\{t\textless \tau \textless T\}}(\theta_{\tau}+1_{\{b\}}1_{\{\tau=\tau^{2}\}}LGD_{I}(V_{\tau_2}-C_{\tau_2})^{+})|\mathcal{G}_t]  	   \end{align}

 In the above, \begin{equation*} \theta=\varepsilon_t-1_{\{\tau=\tau^{1}\}}LGD_C(\varepsilon_t-C_t)^{+}+1_{\{b\}}1_{\{\tau=\tau^{2}\}}LGD_{I}(\varepsilon_t-C_t)^{-} \end{equation*} 
   
 and $A$ is an $\mathbb{F}$ adapted finite variation process equal to \begin{equation}\label{A} A_t=\pi_t + (f_t-c_t)C_u+(r_t-f_t)V_t+(r_t-h_t)H_t.\end{equation}

We begin with an assumption and then derive the new equation for the valuation process $V$ under our assumptions. 

\begin{assumption}\label{A1}
Assume that the martingale $Z_t=\mathbb{E}[\Phi(S_T)\e^{-\int_{0}^{T}(\lambda_s+r_s)ds}|\mathbb{G}_t]$ is strongly orthogonal in the $L^{2}$ sense to the $\mbg$ martingale $M_t=1_{\{\tau \leq t\}}-\int_{0}^{t \wedge \tau}\lambda_sds.$  
\end{assumption} 
\bigskip

From Lemma~\ref{t5}, we can write the first term in Equation~\eqref{contract} as 
\begin{equation}
1_{\{\tau \textgreater t\}}\mathbb{E}[\Phi(S_T)\e^{-\int_{t}^{T}(r_u+\lambda_u)}du|\mathbb{G}_t].
\end{equation} 

Using Proposition~\ref{t5}, we can write the second term in Equation~\eqref{contract} as 

\begin{equation} 
1_{\{\tau \textgreater t\}}\mathbb{E}[\int_{t}^{T}\e^{-\int_{t}^{s}(r_u+\lambda_u)du}A_sds|\mathbb{G}_t] 
\end{equation} 

Therefore, the value of the contract satisfies 
\begin{eqnarray}\label{contract2}
V_t&=& 1_{\{\tau \textgreater t\}}\mathbb{E}[\Phi(S_T)\e^{-\int_{t}^{T}(r_u+\lambda_u)}du|\mathbb{G}_t]+1_{\{\tau \textgreater t\}}\mathbb{E}[\int_{t}^{T}\e^{-\int_{t}^{s}(r_u+\lambda_u)du}A_sds|\mathbb{G}_t] \notag\\
&&\mathbb{E}[\e^{-\int_{t}^{\tau}r_sds}1_{\{t\textless \tau \textless T\}}(\theta_{\tau}+1_{\{b\}}1_{\{\tau=\tau^{2}\}}LGD_{I}(V_{\tau_2}-C_{\tau_2})^{+})|\mathcal{G}_t] . 
\end{eqnarray}

We denote by $G^{1}(t,t)$ the $\mathbb{F}$ submartingale $P(\tau \textgreater t |\mathbb{F}_t).$ Note that, from equation~\eqref{e3}, we have 
\begin{equation}\label{surv} 
G^{1}(t,t)=\e^{-(\int_{0}^{t}\lambda^{1}_s+\lambda^{2}_sds+\int_{0}^{t}\lambda^{1}_sds \vee \int_{0}^{t}\lambda^{2}_sds)}
.\end{equation}  

We remark that under the conditional independence assumption of [BFP], the process 
$$
G(t,t)=e^{-(\int_0^t\gl^1_s+\gl^2_sds)}. 
$$

We also have the following result for the $\mathbb{G}$ compensator of the intensity process $1_{\{\tau \leq t\}}:$ 

\bigskip

\begin{proposition} 

 \begin{equation*}1_{\{\tau \leq t\}}-\int_{0}^{t \wedge \tau}\lambda_sds \end{equation*} is a $\mathbb{G}$ martingale, where \begin{equation}\lambda_t=\frac{\lambda^{1}_t+\lambda^{2}_t}{G^{1}(t,t)}. \end{equation} 

\end{proposition}

\begin{proof} Let $G^{1}(t,t)=\mu_t-C_t$ be the Doob-Meyer decomposition of the $\mathbb{F}$ submartingale $G^{1}(t,t),$ where $C$ is an $\mathbb{F}$ adapted increasing process, and $\mu$ is an $\mathbb{F}$ martingale. From Proposition 3.14 in Bielecki et al (2009), we have that the $\mathbb{G}$ compensator of the process $H_t=1_{\{\tau \leq t\}}$ is given by $\frac{c_t}{G^{1}(t,t)},$ where $c$ is the process such that $dC_t=c_tdt.$ In this case, it is clear that $c_t=\lambda^{1}_t+\lambda^{2}_t.$ \end{proof}

We have the following lemma, whose proof can be found in  Bielicki, Jeanblanc, and Rutkowski (2009). 

\begin{lemma}\label{ph1} Let $Y$ be an integrable random variable. Then, the following holds: \begin{align} 1_{\{ \tau \textgreater t\}}E[Y | \mathbb{G}_t]=\frac{1_{\{\tau \textgreater t\}}   E[Y1_{\{\tau \textgreater t\}}|\mathbb{F}_t]    }{P[(\tau \textgreater t)|\mathbb{F}_t]}\end{align}   \end{lemma}

We also have from Lemma 3.8.1 in Bielecki et al (2009), that the third term in equation~\eqref{contract2} can be represented as 

\begin{equation*}\frac{1}{G^{1}(t,t)}1_{\{\tau \textgreater t\}}\mathbb{E}[\int_{t}^{T}\e^{-\int_{t}^{u}r_sds}(\tilde{\theta}_u+1_{\{b\}}\lambda^{2}_uLGD_{I}(V_u-C_u)^{+})|\mathbb{F}_t].\end{equation*}

Here, \begin{equation}\label{tildet}\tilde{\theta}_t=G^{1}(t,t)(\lambda_t\varepsilon_t-\lambda^{1}_tLGD_C(\varepsilon_t-C_t)^{+}+1_{\{b\}}\lambda^{2}_tLGD_{I}(\varepsilon_t-C_t)^{-}).\end{equation} 

Using Lemma~\ref{ph1}, we can represent the value of the contract in terms of the default-free filtration $\mathbb{F}:$	
\begin{eqnarray}\label{df} 
V_t&=&\frac{1}{G^{1}(t,t)}E[ \Phi(S_T)\e^{-\int_{t}^{T}(r_s+\lambda_s)ds}1_{\{\tau \textgreater t\}}|\mathbb{F}_t] + \frac{1}{G^{1}(t,t)} E[\int_{t}^{T} \e^{-\int_{t}^{u}(r_s+\lambda_s)ds}A_udu|\mathbb{F}_t] \notag\\&& + \frac{1}{G^{1}(t,t)}1_{\{\tau \textgreater t\}}\mathbb{E}[\int_{t}^{T}\e^{-\int_{t}^{u}r_sds}(\tilde{\theta}_u+1_{\{b\}}LGD_{I}\lambda^{2}_u(V_u-C_u)^{+}|\mathbb{F}_t] 
\end{eqnarray}

We proceed to derive a BSDE satisfied by $V_t:$ 

From equation~\eqref{df}, we see that 
\begin{eqnarray*} 
e^{-\int_{0}^{t}(r_s+\lambda_s)ds}V_t&+&\frac{1}{G^{1}(t,t)}\int_{0}^{t}\e^{-\int_{0}^{u}(r_s+\lambda_s)}(A_u+\tilde{\theta}_u+1_{\{b\}}LGD_{I}\lambda^{2}_s(V_u-C_u)^{+})du\\
&=&\frac{1}{G^{1}(t,t)}\mathbb{E}[\int_{0}^{T}\e^{-\int_{0}^{u}(r_s+\lambda_s)ds}(A_u+\tilde{\theta}_u+1_{\{b\}}LGD_{I}\lambda^{2}_u(V_u-C_u)^{+})du|\mathbb{F}_t]\\ 
&&+\frac{1}{G^{1}(t,t)}\mathbb{E}[\e^{-\int_{0}^{T}(\lambda_s+r_s)ds}\Phi(S_T)|\mathbb{F}_t], 
\end{eqnarray*}

Let us denote by $M^{1}$ the $\mathbb{F}$ local martingale \begin{equation*} \mathbb{E}[\int_{0}^{T}\e^{-\int_{0}^{u}(r_s+\lambda_s)ds}(A_u+\tilde{\theta}_u+1_{\{b\}}LGD_{I}\lambda^{2}_u(V_u-C_u)^{+})du|\mathbb{F}_t]\end{equation*} and by $M^{2}$ the $\mathbb{F}$ local martingale \begin{equation*}\mathbb{E}[\e^{-\int_{0}^{T}(r_s+\lambda_s)ds}\Phi(S_T)|\mathbb{F}_t].\end{equation*}

Then, we have \begin{align}\label{BSDE1}e^{-\int_{0}^{t}(r_s+\lambda_s)ds}dV_t-(r_t+\lambda_t)\e^{-\int_{0}^{t}(r_s+\lambda_s)ds}V_tdt+\frac{1}{G^{1}(t,t)}\e^{-\int_{0}^{t}(r_s+\lambda_s)ds}(A_t+\tilde{\theta}_t)dt \\ \nonumber +\int_{0}^{t}\e^{-\int_{0}^{u}(r_s+\lambda_s)}(A_u+\tilde{\theta}_u)dud(\frac{1}{G^{1}(t,t)}) -(M^{1}_t+M^{2}_t)d(\frac{1}{G^{1}(t,t)})&=\frac{dM^{1}_t+dM^{2}_t}{G^{1}(t,t)}.\end{align} 

Now, $M^{1}$ and $M^{2}$ are closed $\mathbb{F}$ martingales and thus $\int_{0}^{t}\frac{1}{G^{1}(s,s)}dM^{1}_s$ and $\int_{0}^{t}\frac{1}{G^{1}(s,s)}dM^{2}_s$ are $\mathbb{F}$ local martingales.

Therefore, by the martingale representation theorem, we have that $$\int_{0}^{t}\frac{1}{G^{1}(s,s)}(dM^{1}_s+dM^{2}_s)=\int_{0}^{t}Z_sdW_s,$$ for an $\mathbb{F}$ predictable process $Z.$

Therefore, we have the following BSDE for $V:$ 
\begin{align}\label{bsde}
dV_t&=(r_t+\lambda_t)V_tdt-\frac{1}{G^{1}(t,t)}(A_t+1_{\{b\}}LGD_{I}\lambda^{2}_t(V_t-C_t)^{+}+\tilde{\theta}_t)dt \\& \nonumber-\e^{\int_{0}^{t}(r_s+\lambda_s)ds}\int_{0}^{t}\e^{-\int_{0}^{u}(r_s+\lambda_s)}(A_u+\tilde{\theta}_u+1_{\{b\}}LGD_{I}\lambda^{2}_u(V_u-C_u)^{+})dud(\frac{1}{G^{1}(t,t)})\\& \nonumber +\e^{\int_{0}^{t}(r_s+\lambda_s)ds}(M^{1}_t+M^{2}_t)d(\frac{1}{G^{1}(t,t)})+ \bar{Z}_tdW_t,
\end{align} 
where $\bar{Z}_t=\e^{\int_{0}^{t}(r_s+\lambda_s)ds}Z_t.$ 

Note that 
\begin{equation*}d(\frac{1}{G^{1}(t,t)})=\frac{1}{G^{1}(t,t)}(\frac{3}{2}(\lambda^{1}_t+\lambda^{2}_t)+\frac{1}{2}(\lambda^{1}_t-\lambda^{2}_t)\sign{(\lambda^{1}_t-\lambda^{2}_t)})dt,   
\end{equation*} 
which makes equation~\eqref{bsde}  rather

\begin{eqnarray}\label{bsde'} 
dV_t&=&[(r_t+\lambda_t)V_t-\frac{(A_t+\tilde{\theta}_t+1_{\{b\}}LGD_{I}\lambda^{2}_t(V_t-C_t)^{+})}{G^{1}(t,t)} \\&& \nonumber -\e^{\int_{0}^{t}(r_s+\lambda_s)ds}\int_{0}^{t}\e^{-\int_{0}^{u}(r_s+\lambda_s)}(A_u+\tilde{\theta}_u+1_{\{b\}}LGD_{I}\lambda^{2}_u(V_u-C_u)^{+})du\frac{1}{G^{1}(t,t)}(\frac{3}{2}(\lambda^{1}_t+\lambda^{2}_t)\\
&&+\frac{1}{2}(\lambda^{1}_t-\lambda^{2}_t)\sign(\lambda^{1}_t-\lambda^{2}_t)) -\e^{\int_{0}^{t}(r_s+\lambda_s)ds}(M^{1}_t+M^{2}_t)\notag\\
&&\times\frac{1}{G^{1}(t,t)}(\frac{3}{2}(\lambda^{1}_t+\lambda^{2}_t)+\frac{1}{2}(\lambda^{1}_t-\lambda^{2}_t)\sign(\lambda^{1}_t-\lambda^{2}_t))]dt+ \bar{Z}_tdW_t. 
\end{eqnarray}

\section{When do we have the orthogonality condition?}\label{s4bis}

We show the applicability of our results to the question at hand. To wit: Does Assumption~\ref{A1} hold for our Bivariate Exponential $(Z^1,Z^2)$ whose survival function is given by $P(Z^1>s,Z^2>t)=\exp\{-s-t-\bga\max(s,t)\}$ as in~\eqref{3bis}? We denote
\bee\label{ne}
Y=\Phi(S_T)\e^{-\int_{0}^{T}(\lambda_s+r_s)ds}
\eee
We make two quite mild assumptions:

\bigskip

\noindent {\bf {\large Assumptions:}}

\begin{itemize}

\item[a)] $Y$ in~\eqref{ne} is in $L^1$

\item[b)] The processes $\gl^1$ and $\gl^2$ are both bounded
\end{itemize}


\bigskip

\noindent
We let $(Y_t)_{t\geq 0}$ be the $\mbg$ uniformly integrable martingale $Y_t=E\{Y\vert\mcg_t\}$. Recall that
\bee\label{pep1}
M_t=1_{\{t\geq\tau\}}-\int_0^{t\wedge\tau}\gl_sds
\eee
and we want to show that the process $[Y,M]$ is a $\mbg$ martingale. Here $Y$ refers at times to the martingale $Y=(Y_t)_{t\geq 0}$ and at other times to the random variable  $Y=\Phi(S_T)\e^{-\int_{0}^{T}(\lambda_s+r_s)ds}$. It should be clear which is which from the context. Using the theory of the expansion of filtrations, and in particular by~\cite[Theorem 14 of Chapter VI]{PP} we have
\bee\label{pep2}
Y_t=E(Y\vert\mcg_t)=\frac{1}{U_t}(^o(Y1_{\{t<\tau\}}))+Y1_{\{t\geq\tau\}}
\eee
In~\eqref{pep2} by $^o(R)$ we mean the optional projection of the process $R$ onto the $\mbf$ filtration. Also recall that we defined  $U_t=^o1_{\{\tau>t\}}$ earlier. 

The process $[Y,M]=\Delta Y_\tau\Delta M_\tau$, since $M$ has only one jump, and it's at $\tau$. Since $\Delta M_\tau=1$, to show that $[Y,M]$ is a martingale, it suffices to show that $E([Y,M]_\nu)=0$ for every stopping time $\nu$. Due to the simple nature of $[Y,M]$, i.e., that it jumps only at the time $\tau$, it suffices to show that $E(\Delta Y_\tau)=0.$ We have trivially:
\bee\label{ne3}
E(\Delta {Y_\tau})=E(Y_\tau-Y_{\tau-})=E(Y_\tau)-E(Y_{\tau-})
\eee

Since $Y$ is a uniformly integrable martingale in $\mbg$ and $\tau$ is a $\mbg$ stopping time, we have 
$$
E(Y_\tau)=E(Y)
$$
Therefore we need to investigate $E(Y_{\tau-})$ and to show it is equal to $E(Y)$. {\color{blue} By uniform integrability in $t$ of $\frac{1}{U_t}(^o(Y1_{\{t<\tau\}})),$} we have 
$$ E(Y_{\tau-})=E(\frac{1}{U_t}(^o(Y1_{\{t<\tau\}}))\big |_{\{t=\tau-\}})=E(\lim_{t\nearrow\infty}\frac{1}{U_t}(^o(Y1_{\{t<\tau\}}))).$$

By~\eqref{pep2} we need to show that $E(\frac{1}{U_t}(^o(Y1_{\{t<\tau\}})))=E(Y)$. At a given time $t$, the optional projection is equal to the conditional expectation a.s., so we have
$$
E(\frac{1}{U_t}(^o(Y1_{\{t<\tau\}}))=E(\frac{1}{U_t}E((Y1_{\{t<\tau\}}))\vert\mcf_t))=E(\frac{1}{U_t}Y1_{\{t<\tau\}})
$$
Let $\mu$ be the distribution measure of $(Z_1,Z_2)$, our bivariate exponential. Recall that our stopping time in question is $\tau=\tau_1\wedge\tau_2$.

The event 
\bee\label{em1}
\{t<\tau\}=\{t<\tau_1\}\cap\{t<\tau_2\}=(\int_0^t\gl^1_sds\textless Z_1)\cap(\int_0^t\gl^2_sds\textless Z_2)
\eee

Using this, and taking expectations of the conditional expectations, we have
\begin{eqnarray}\label{e6ter}
E(Y1_{\{t<\tau\}})&=&E(E(Y1_{\{t<\tau\}}\vert\mcf_t))=E(E(Y1_{\{\int_0^t\gl^1_sds\textless Z_1\}}1_{\{\int_0^t\gl^2_sds\textless Z_2\}}\vert\mcf_t))\\
&=&E(\int_0^\infty E(Y1_{\{\int_0^t\gl^1_sds\textless Z_1\}}1_{\{\int_0^t\gl^2_sds\textless Z_2\}}\vert (Z_1=x,Z_2=u)\mu(dxdu)\vert\mcf_t))\notag\\
&=&E(\int_0^\infty\int_0^\infty E(Y1_{\{\int_0^t\gl^1_sds\textless x\}}1_{\{\int_0^t\gl^2_sds\textless u\}}\vert (Z_1=x,Z_2=u)\vert\mcf_t)\mu(dxdu)\notag\\
&=& \text{and using the independence of }(Z_1,Z_2)\text{ from }\mbf\notag\\
&&E(\int_0^\infty \Theta_t1_{\{\int_0^t\gl^1_sds\textless x\}}1_{\{\int_0^t\gl^2_sds\textless u\}}\mu(dxdu)),\text{ where }\Theta_t=E(Y\vert\mcf_t)
\end{eqnarray}

In the above we also used Fubini's Theorem for conditional expectations. The key term, using~\eqref{e6ter} , computes as follows:
\bee\label{ne11}
E(\frac{1}{U_t}{^o(Y1_{\{t<\tau\}})})=E(\oneut\Theta_t1_{\{t<\tau\}})\text{ with }\Theta_t\in\mcf_t
\eee
Therefore
\begin{eqnarray}\label{e14}
E(\frac{1}{U_{t}}(\Theta_tE(1_{\{t<\tau\}}\vert\mcf_t)))&=&E(\frac{1}{^o1_{\{\tau>t\}}}\Theta_t ^o1_{\{t<\tau\}})\notag\\
&=&E(\frac{ ^o1_{\{t<\tau\}}}{ ^o1_{\{t<\tau\}}}\Theta_t)=E(E(Y\vert\mcf_t)\frac{ ^o1_{\{t<\tau\}}}{ ^o1_{\{t<\tau\}}})\notag\\
&=&E(Y\frac{ ^o1_{\{t<\tau\}}}{ ^o1_{\{t<\tau\}}})
\end{eqnarray}

We conclude that 
\bee\label{ne17}
E(Y_{\tau-})=E(\lim_{t\nearrow\infty}\frac{1}{U_t}(^o(Y1_{\{t<\tau\}})))=\lim_{t\nearrow\infty}E(\frac{1}{U_t}(^o(Y1_{\{t<\tau\}}))=\lim_{t\nearrow\infty}E(Y)=E(Y)
\eee

As an addendum, we note that $0<\oneut<\infty$ as a consequence of our formula for $G(t,t)$, given in~\eqref{surv}. The assumption that $\gl^1$ and $\gl^2$ are bounded gives us that $\oneut$ is also bounded. 

\qed

\section{The FBSDE}\label{s5}

Here we are in an easy case for FBSDEs, since the forward equation and the backward equation are uncoupled. The fact that we are working on a specified compact time interval makes the analysis even simpler. Of course, it is a standard key assumption that $\gs$ of the forward equation never vanish. we can treat the two equations separately. For the forward equation, we only need the usual Lipschitz and linear growth conditions that are classical (see for example Protter (2005)). After we have the solution $S$ and hence have $S_T$ and this $\Phi(S_T)$, we can invoke any of several theorem about the existence of a solution of a Backwards SDE, and we need not get lost in the thicket of the theory of BSDEs. 

For example, Theorem 1 of Lepeltier and San Martin (1997) gives existence under minimal hypothese on the coefficient (just continuity) but if one also wants uniqueness of the solution (and we do), then we need more. This is provided by the now classical and seminal paper of Pardoux and Peng (1990). Thus we have:

\bigskip

\begin{theorem}\label{t2}
Let $(X,Y,Z)$ solve the following on the interval $[0,T]:$ 

 \begin{eqnarray}\label{gen}
dX_t&=&b(t,X_t)dt+ \sigma(t,X_t)dW_t;  \qquad X_0=x \\ \label{gen1} 
dY_t&=&-f(t,X_t,Y_t,Z_t)dt+Z_tdW_t; \quad Y_T=g(X_T).	 \end{eqnarray}

Assume that the functions $b,$ $\sigma$ and $f$ satisfy, for a constant $K$ and for all $t,$

\begin{itemize}

\item  $|b(t, x)-b(t, x^{'})| + |\sigma(t, x)-\sigma(t, x^{'})| \leq K|x-x^{'}|$ 

\item  $|b(t, x)| + |\sigma(t, x)|\leq K(1 + |x|)$

\item $ |f(t, x, y, z)-f(t, x, y^{'}, z^{'})| \leq K(|y-y^{'}| + |z-z^{'}|).$

\end{itemize}

Assume, in addition, that there exists a positive constant $p \geq \frac{1}{2}$ such that 
\begin{equation*} |g(x)|+| f(t,x,0,0)| \leq K(1+p|x|). \end{equation*}

Then the FBSDE~\eqref{gen}-\eqref{gen1} has a unique solution.
\end{theorem}





In the more complicated case of uncoupled FBSDEs, J. Ma, P. Protter, and J. Yong (1994) developed the so-called four step scheme approach to solving SDEs. They connected the solution to the existence of classical solutions of certain hyperbolic PDEs, and under the right hypotheses were able to show that the backwards solution was a function of a partial derivative of the solution of said PDE. This is a nice feature, explicitly giving the backwards SDE in terms of the forward SDE. It has now been extended to cover different situations. We use the following version:

\bigskip

\begin{theorem}\label{t3}
Consider the following FBDSE:

\begin{eqnarray}\label{fb} dS_t&=&r_tS_tdt+\sigma(S_t,t)dW_t; \qquad S_0=s_0 \\ \label{fbv}dV_t&=&[(r_t+\lambda_t)V_t+\frac{(A_t+\tilde{\theta}_t+1_{\{b\}}LGD_{I}\lambda^{2}_t(V_t-C_t)^{+})}{G^{1}(t,t)} \\
&&\nonumber+\e^{\int_{0}^{t}(r_s+\lambda_s)ds}\int_{0}^{t}\e^{-\int_{0}^{u}(r_s+\lambda_s)}(A_u+\tilde{\theta}_u+1_{\{b\}}LGD_{I}\lambda^{2}_t(V_u-C_u)^{+})du(\frac{3}{2}(\lambda^{1}_t+\lambda^{2}_t)\notag\\
&&+\frac{1}{2}(\lambda^{1}_t-\lambda^{2}_t)\sign(\lambda^{1}_t-\lambda^{2}_t)) -\e^{\int_{0}^{t}(r_s+\lambda_s)ds}(M^{1}_t+M^{2}_t)\notag\\
&&\times(\frac{3}{2}(\lambda^{1}_t+\lambda^{2}_t)+\frac{1}{2}(\lambda^{1}_t-\lambda^{2}_t)\sign(\lambda^{1}_t-\lambda^{2}_t))]dt+ \bar{Z}_tdW_t; \quad V_T=\Phi(S_T)\notag.
 \end{eqnarray} 

Let us denote by $K=(K_t)_{t\geq 0}$ the bounded process 
$$
\frac{3}{2}(\lambda^{1}_t+\lambda^{2}_t)+\frac{1}{2}(\lambda^{1}_t-\lambda^{2}_t)\sign(\lambda^{1}_t-\lambda^{2}_t).
$$

Assume the following: 
\begin{enumerate}		
\item $\lambda^{1},$  $\lambda^{2}$ are bounded, Lipschitz continuous functions of $S.$  
\item $r,$ $h,$ and $f$ are all bounded, deterministic functions of $t.$
\item $\pi_t$ is a deterministic function of $t$ and $S_t,$ Lipschitz continuous in $S.$ 
\item $C_t=\alpha_t V_t,$ where $0 \leq \alpha \leq 1$ 
\item The closeout value $\varepsilon_t$ is equal to $V_t.$
\item $H_t=H(t, S_t,V_t, Z_t),$ $H(t,x,v,z)$ being a deterministic Lipschitz-continuous function in v, z, uniformly in t. Also, $H(t,x,0,0)$ is continuous in $x.$ 
\item There exists a constant $L$ and a constant $p \geq \frac{1}{2}$ such that $|g(x)|+ | B(t,x,0,0)| \leq L(1+p|x|).$
\end{enumerate} 
Then, the FBSDE~\eqref{fb}-\eqref{fbv} has a unique solution; that is, there exist deterministic functions $u(t,x)$ and $d(t,x)$ such that the solution $(S_t, V_t, Z_t)$ is given by $V_t=u(t,S_t),$ and $Z_t=d(t,S_t)\sigma(t,S_t).$  $u(t, x)=V^{x}_t$ is the unique viscosity solution to the following partial differential equation: 
 \begin{eqnarray}\frac{\partial u}{\partial t}+\frac{1}{2}\sigma^{2}(t,x) \frac{\partial^2 u}{\partial x^2}+r_tx\frac{\partial u}{\partial x}+B(t,x,u,\sigma\frac{\partial u}{\partial x})&=&0 \\ u(T,x)&=&\Phi(x). \end{eqnarray} 
 
\end{theorem} 

\begin{proof} 
The result follows from Theorem~\ref{t2} plus now classic results on FBSDEs, as detailed in, for example, Ma-Protter-Yong (1994) and also J. Zhang (2006). Therefore we only need to show that the drift function in~\eqref{fbv}, $B(t,x,u,z),$ given by \begin{align}\label{B} B(t,x,u,z)&=r_tV_t-\frac{A_t+\tilde{\theta}_t}{G^{1}(t,t)}-\frac{1}{G^{1}(t,t)}1_{\{b\}}LGD_{I}\lambda^{2}_t(1-\alpha_t)V_t \\& \nonumber -K_t\e^{\int_{0}^{t}(r_s+\lambda_s)ds}\int_{0}^{t}\e^{-\int_{0}^{u}(r_s+\lambda_s)}(A_u+\tilde{\theta}_u+1_{\{b\}}LGD_{I}\lambda^{2}_u(V_u-C_u)^{+})du) \\& \nonumber -K_t\e^{\int_{0}^{t}(r_s+\lambda_s)ds}(M^{1}_t+M^{2}_t) ,\end{align} satisfies $|B(t,x,u,z)-B(t,x,u^{'},z^{'})| \leq L(|u-u^{'}|+|z-z^{'}|),$. For this, it is sufficient to note that it is continuous and is a sum of bounded functions of $t$, piece-wise linear in $u$ (hence Lipschitz continuous in $u$ since the domain is compact) and piecewise-linear as a function of $H,$ which is a deterministic Lipschitz-continuous function in v, z, uniformly in t. 

\end{proof}

\section{A Comparison of Two FBSDEs and Conclusion}

We have extended the results of Brigo et al (2019) to the case where we have replaced conditional independence with an orthogonality condition. Therefore it is of interest to compare the BFSDE derived by Brigo et al with the new BFSDE we have obtained under our hypotheses.  The equation of Brigo at al (2019) is equation $(7)$ on page 795 of their paper. to wit, it is:

\begin{eqnarray}\label{df'}dS^{x}_t&=&r_rS^{x}_tdt+\sigma(S^{x}_t)dW_t; \qquad 0 \textless t \leq T\\ \nonumber 
S_0&=&x; \\ \nonumber
dV^{x}_t&=&-(\pi_t+\tilde{\theta}_t+((1-\alpha_t)(1_{\{b\}}LGD_{I}1_{\{V^{x}_t \textgreater 0\}}\lambda^{2}_t-f_t)-\lambda_t-c_t\alpha_t)V^{x}_t-(r_t-h_t)H_t)dt \\ \label{dbs}&&+ Z^{x}_tdW_t \\  \nonumber V^{x}_T&=&\Phi(S^{x}_T). 
\end{eqnarray} 

The difference in equations~\eqref{df'}-\eqref{dbs} and equations~\eqref{fb}-\eqref{fbv} lies in the drift of the backward stochastic differential component of the FBSDEs: the drift of~\eqref{dbs}, call it $\bar{B}_t,$ is given by \begin{equation}\label{BB} -(\pi_t+\tilde{\theta}_t+((1-\alpha_t)(1_{\{b\}}LGD_{I}1_{\{V^{x}_t \textgreater 0\}}\lambda^{2}_t-f_t)-\lambda_t-c_t\alpha_t)V^{x}_t-(r_t-h_t)H_t). \end{equation}  

Recalling the definitions of $A_t$ and $\tilde{\theta}_t$ in equations~\eqref{A} and~\eqref{tildet}, respectively, and comparing this to~\eqref{B}, we see that in~\eqref{B}, the term $$-(A_t+\tilde{\theta}_t+ 1_{\{b\}}LGD_{I}\lambda^{2}_t(1-\alpha_t)V_t)$$ is divided by $\frac{1}{G^{1}(t,t)},$ and that ~\eqref{B} contains two supplementary terms, namely, $$-K_t\e^{\int_{0}^{t}(r_s+\lambda_s)ds}\int_{0}^{t}\e^{-\int_{0}^{u}(r_s+\lambda_s)}(A_u+\tilde{\theta}_u+1_{\{b\}}LGD_{I}\lambda^{2}_u(V_u-C_u)^{+})du)$$ and $$-K_t\e^{\int_{0}^{t}(r_s+\lambda_s)ds}(M^{1}_t+M^{2}_t).$$

Thus we see that replacing the hypothesis of conditional independence with the weaker one of orthogonality gives a different BSDE modeling the credit risk in the CVX framework, the difference occurring in the drift, and we have made this difference explicit. We consider such a weakening of the hypotheses to be of value because in this more general context we allow for the modeling of simultaneous defaults, with ``defaults" here meaning credit events. An example might be some major company (such as Boeing aircraft) has an issue that created a sudden lack of demand from suppliers for their parts (think the Boeing 737 Max 8 jet airplane here). The suppliers will have cascading credit problems because of the same cause, and some might be best be modeled as simultaneous, a situation not covered by the assumption of conditional independence given the market. In such a model simultaneous credit events are not allowed. 
..

\end{document}